\documentclass[10pt, letterpaper]{article}
\usepackage{fullpage}
\usepackage{amsmath, amssymb,amsthm}

\newtheorem{theorem}{Theorem}[section]

\newtheorem{lemma}[theorem]{Lemma}
\newtheorem{proposition}[theorem]{Proposition}
\newtheorem{corollary}[theorem]{Corollary}
\numberwithin{equation}{section}

\newenvironment{definition}[1][Definition]{\begin{trivlist}
\item[\hskip \labelsep {\bfseries #1}]}{\end{trivlist}}

\newenvironment{remark}[1][Remark]{\begin{trivlist}
\item[\hskip \labelsep {\bfseries #1}]}{\end{trivlist}}

\def\P{{\mathbb P}}        % probability
\def\E{{\mathbb E}}        % expectation
\def\Z{{\mathbb Z}}         % integers
\def\F{{\cal F}}                 % sigma field
\def\N{{\mathbb N}}       % naturals
\def\R{{\mathbb R}}       % reals
\def\1{{\mathbf 1}}         % indicator
\def\B{{\cal B}}               % ditto
\def\Var{{\mathbf {Var}\,}}    % variance

\title{An $L^1$ Ergodic Theorem for Sparse Random Subsequences}
\author{Patrick LaVictoire}

\begin{document}
\maketitle
\begin{abstract}
We prove an $L^1$ subsequence ergodic theorem for sequences chosen by independent random selector variables, thereby showing the existence of universally $L^1$-good sequences nearly as sparse as the set of squares.  In the process, we prove that a certain deterministic condition implies a weak maximal inequality for a sequence of $\ell^1$ convolution operators (Prop. \ref{shine on}).
\end{abstract}
\section{Introduction}
Let $(X,\F,m)$ be a non-atomic probability space and $T$ a measure-preserving transformation on $X$; we call $(X,\F,m,T)$ a dynamical system.  For a sequence of integers ${\mathfrak n}=\{n_k\}$ and any $f\in L^1(X)$, we may define the subsequence average
$$A_N^{({\mathfrak n})}f(x):= \frac1N \sum_{k=1}^N f(T^{n_k}x).$$
Given a sequence ${\mathfrak n}$, a major question is for which $1\leq p\leq\infty$ and which $(X,\F,m,T)$ we have convergence of various sorts for all $f\in L^p(X)$.  An important definition along these lines is as follows:
\begin{definition}
A sequence of integers ${\mathfrak n}=\{n_k\}$ is \emph{universally $L^p$-good} if for every dynamical system $(X,\F,m, T)$ and every $f\in L^p(X,m)$, $\displaystyle\lim_{N\to\infty} A_N ^{({\mathfrak n})} f(x)$ exists for almost every $x\in X$.\\ 
\end{definition}
\noindent Birkhoff's Ergodic Theorem asserts, for instance, that the sequence $n_k=k$ is universally $L^1$-good.  On the other extreme, the sequence $n_k=2^k$ is not even universally $L^\infty$-good (lacunary sequences are bad for convergence of ergodic averages in various strong ways, see for example \cite{JR2} or \cite{SSO}).  Between these extrema lie many results on the existence of universally $L^p$-good sequences of various sorts, beginning with Bourgain's celebrated result \cite{JB2} that $n_k=k^2$ is universally $L^2$-good; see \cite{JB3} and \cite{EAS} for extensions of this result to other sequences.
\\
\\ The most restrictive case $p=1$ is more subtle than the others.  A surprising illustration of the difference is the recent result of Buczolich and Mauldin that $n_k=k^2$ is not universally $L^1$-good \cite{DSA}.  Positive results in $L^1$ have been difficult to come by, particularly for sequences which are sparse in $\N$.
\\
\\ Universally $L^1$-good sequences of density 0 had long been known to exist, but these were sparse block sequences, which consist of large 'blocks' of consecutive integers, separated by wide gaps. Bellow and Losert \cite{BL} showed that for any $F: \N\to\R^+$, there exists a universally $L^1$-good block sequence $\{n_k\}$ with $n_k\geq F(k)$.  To distinguish such block sequences from more uniformly distributed ones, we recall the notion of Banach density:
\begin{definition}
A sequence of positive integers $\{n_k\}$ has Banach density $c$ if $$\lim_{m\to\infty} \sup_{N} \frac{|\{n_k\in[N,N+m)\}|}{m}=c.$$
\end{definition}
Note that block sequences with arbitrarily large block lengths have Banach density 1 (the sequences in \cite{BL} are all of this sort).  The first example of a Banach density 0 universally $L^1$-good sequence was constructed by Buczolich \cite{GTI}, and Urban and Zienkiewicz \cite{UZ} subsequently proved that the sequence $\lfloor k^a \rfloor$ for $1<a<1+\frac1{1000}$ is universally $L^1$-good.
\\
\\ Bourgain \cite{JB2} noted that certain sparse random sequences were universally $L^p$-good with probability 1 for all $p>1$.  These sequences are generated as follows: given a decreasing sequence of probabilities $\{\tau_j:j\in\N\}$, let $\{\xi_j:j\in\N\}$ be independent random variables on a probability space $\Omega$ with $\P(\xi_j=1)=\tau_j,\;\P(\xi_j=0)=1-\tau_j$.  Then for each $\omega\in\Omega$, define a random sequence by taking the set $\{n:\xi_n(\omega)=1\}$ in increasing order.  (For $\alpha>0$ and $\tau_j=O(j^{-\alpha})$, these sequences have Banach density 0 with probability 1; see Prop. \ref{banach} of this paper.)
\\
\\In their treatment \cite{PETHA} of Bourgain's method, Rosenblatt and Wierdl demonstrate by Fourier analysis that if $\tau_j\to0$ slowly enough (e.g. $\tau_j\geq \frac{c(\log \log j)^{1+\epsilon}}{j}$ suffices), then $\{n:\xi_n(\omega)=1\}$ is universally $L^2$-good with probability 1 (see Example 4.7), thus proving the existence of superpolynomial universally $L^2$-good sequences.  However, their approach cannot be applied to the $L^1$ case.
\\
\\ In this paper, we apply a construction of \cite{UZ} to these random sequences and achieve the following $L^1$ result:

\begin{theorem}
\label{L1}
Let $0<\alpha<1/2$, and let $\xi_n$ be independent selector variables on $\Omega$ with $\P(\xi_n=1)=n^{-\alpha}$.  Then there exists a set $\Omega'\subset\Omega$ of probability 1 such that for every $\omega\in\Omega'$, $\{n:\xi_n(\omega)=1\}$ is universally $L^1$-good.
\end{theorem}

\noindent Thus we prove the existence of universally $L^1$-good sequences which grow more rapidly than the ones obtained in \cite{UZ} or \cite{GTI}, and which grow uniformly as compared to the sparse block sequences of \cite{BL}.  In particular, with probability 1 these sequences have $n_k={\mathbf\Theta}(k^{1/(1-\gamma)})$ (that is, $c_\omega k^{1/(1-\gamma)}\leq n_k\leq C_\omega k^{1/(1-\gamma)}$), so Theorem \ref{L1} applies to random sequences nearly as sparse as the sequence of squares.
\\
\\ Our method is as follows: in Section 2 we define our notation and reduce the problem (by transference) to one of proving a weak maximal inequality on $\Z$ for convolutions with a series of random $\ell^1(\Z)$ functions $\mu_n^{(\omega)}$.  In Section 3, we use the framework of \cite{UZ} to prove this inequality under an assumption about the convolutions of $\mu_n^{(\omega)}$ with their reflections about the origin; and in Section 4, we establish that with probability 1, these random functions do indeed satisfy that assumption.

\section{Definitions, and Reduction to a Weak Maximal Inequality}

Let $\{\tau_n:n\in\N\}$ be a nonincreasing sequence of probabilities.  Let $\Omega$ be a probability space, and define independent mean $\tau_n$ Bernoulli random variables $\{\xi_n(\omega): n\in\N\}$ on  $\Omega$; that is, $\P(\xi_n=1)=\tau_n$ and $\P(\xi_n=0)=1-\tau_n$.  Let
$$\beta(N):=\sum_{n=1}^N\tau_n.$$
\begin{definition} For a dynamical system $(X,\F,m, T\})$ and $f\in L^1(X)$, define  the random average 
\begin{eqnarray*}
A_N^{(\omega)}f(x):=\beta(N)^{-1}\sum_{n=1}^N\xi_n(\omega)f(T^nx)
\end{eqnarray*}
 and its $L^1(X)$-valued expectation 
\begin{eqnarray*}
\E_\omega A_N^{(\omega)}f(x):=\beta(N)^{-1}\sum_{n=1}^N\tau_nf(T^nx).
\end{eqnarray*}
\end{definition}
\begin{remark}
$A_N^{(\omega)}f$ differs from the subsequence averages discussed before by the factor $\displaystyle\beta(N)^{-1}\sum_{n=1}^N\xi_n(\omega)$.  However, if $\beta(N)\to\infty$, then with probability 1 in $\Omega$, $\beta(N)^{-1}\sum_{n=1}^N\xi_n(\omega)\to1$; this follows quickly from an application of Chernoff's Inequality, which we will use elsewhere in this paper:
\end{remark}
\begin{theorem}
\label{chernoff}
Let $\{X_n\}_{n=1}^N$ a sequence of independent random variables with $|X_n|\leq 1$ and $\E X_n=0$.  Let $X=\displaystyle\sum_{n=1}^N X_n$, and $\sigma^2=\Var X=\E X^2$.  Then for any $\lambda>0$, $$\P(|X|\geq\lambda\sigma)\leq2\max(e^{-\lambda^2/4},e^{-\lambda\sigma/2}).$$
\end{theorem}
\begin{proof}
This is Theorem 1.8 in \cite{TV}, for example.
\end{proof}

\noindent We restrict ourselves to the set $\Omega_1\subset\Omega$ on which $\beta(N)^{-1}\sum_{n=1}^N\xi_n(\omega)\to1$.  The a.e. convergence of $A_N^{(\omega_0)}f(x)$ for every dynamical system $(X,\F,m,T)$ and every $f\in L^p(X)$ is then equivalent to the statement that $\{j\in\N: \xi_j(\omega_0)=1\}$ is universally $L^p$-good.  We further remark that for a power law $\tau_n=n^{-\alpha}$, we have $N^{\alpha-1}\beta(N)\to C\in(0,\infty)$ for $\alpha<1$.
\\ 
\\ By Bourgain's result in \cite{JB2}, there is a set $\Omega_2\subset\Omega_1$ with $\P(\Omega_2)=1$ such that for $\omega\in\Omega_2$ we have a.e. convergence of $A_N^{(\omega)}f$ for all $f\in L^2(X)$, which is dense in $L^1(X)$.  Theorem \ref{L1} thus reduces to proving on a set of probability 1 the weak maximal inequality
\begin{eqnarray}
\|\sup_N |A^{(\omega)}_Nf|\|_{1,\infty}\leq C_\omega\|f\|_1 \;\forall f\in L^1(X).
\end{eqnarray}
As usual, it is enough to take this supremum over the dyadic subsequence $N\in\{2^j:j\in\N\}$, since $\frac{\beta(2^{j+1})}{\beta(2^j)}\leq 2$ and thus $0\leq A^{(\omega)}_Nf\leq 2A^{(\omega)}_{2^{j+1}}f$ for $f\geq0$ and $2^j\leq N<2^{j+1}$.  As in \cite{APET} and other papers, we can transfer this problem to the group algebra $\ell^1(\Z)$.  Namely, if we define the random $\ell^1(\Z)$ functions
\begin{eqnarray*}
\mu_j^{(\omega)}(n)&:=&\left\{\begin{array}{ll} \beta(2^j)^{-1}\xi_{n}(\omega), & 1\leq n\leq 2^j \\ 0 & \text{otherwise} \end{array}\right.\\
\E\mu_j(n)&:=&\left\{\begin{array}{ll} \beta(2^j)^{-1}\tau_n, & 1\leq n\leq 2^j \\ 0 & \text{otherwise}\end{array}\right.\\
\nu_j^{(\omega)}(n)&:=&\mu_j^{(\omega)}(n)-\E\mu_j^{(\omega)}(n),
\end{eqnarray*}
then $\mu_j^{(\omega)}$ and $\E\mu_j$ correspond to the operators $A_{2^j}^{(\omega)}$ and $\E_\omega A_{2^j}^{(\omega)}$, respectively.  It suffices to prove that with probability 1 in $\Omega$,
\begin{eqnarray}
\label{eclipse}
\| \sup_j | \varphi\ast\mu_{j}^{(\omega)} |\|_{1,\infty}\leq C_\omega\| \varphi\|_1\hspace{12pt}\forall \varphi\in\ell^1(\Z).
\end{eqnarray}
We will use $\tilde\mu$ to denote the reflection of a function $\mu$ about the origin; as the adjoint of the operator given by convolution with $\mu$ is a convolution with $\tilde\mu$, this will be an important object.  (It would be standard to use the notation $\mu^*$, but this becomes unwieldy when using other superscripts as above.)
% BEGIN L^1 SECTION
\section{Calderon-Zygmund Argument}
The proof of (\ref{eclipse}) uses a generalization of a deterministic argument from the paper by Urban and Zienkiewicz \cite{UZ}, related to a construction of Christ in \cite{MC2}:
%DETERMINISTIC PROPOSITION
\begin{proposition}
\label{shine on}
Let $\mu_j$ and $\nu_j$ be sequences of functions in $\ell^1(\Z)$.  Let $r_j:=|\emph{supp }\mu_j|$ and suppose $\emph{supp } \nu_j \subset [-R_j,R_j]$.  Assume there exist $\epsilon>0$ and $A,A_0,A_1<\infty$ such that $\sum_{j\leq k} r_j\leq Ar_k\;\forall k\in\N$ and
\begin{eqnarray}
\label{convo}
|\nu_j\ast\tilde\nu_j(x)|\leq A_0r_j^{-1}\delta_0(x)+A_1R_j^{-(1+\epsilon)},\hspace{12pt}\forall x\in\Z.
\end{eqnarray}
If for all $\varphi\in\ell^1$, $\| \displaystyle\sup_j \varphi\ast |\mu_j-\nu_j|\|_{1,\infty}\leq C\|\varphi\|_1$ and $\| \displaystyle\sup_j|\varphi\ast\mu_j|\|_{p,\infty}\leq C_p\|\varphi\|_p$ for some $1<p\leq\infty,$ then 
\begin{eqnarray}
\label{sup}
\|\sup_j|\varphi\ast\mu_j|\|_{1,\infty}\leq C'\|\varphi\|_1 \hspace{12pt}\forall\varphi\in\ell^1(\Z).
\end{eqnarray}
\end{proposition}
\begin{proof}
We will follow the argument in Section 3 of \cite{UZ}, which makes use of a Calderon-Zygmund type decomposition of $\varphi$ depending on the index $j$.  We begin with the standard decomposition at height $\lambda>0$: $\varphi =g+b$, where 
\begin{itemize}
\item $\|g\|_\infty\leq\lambda$
\item $b=\displaystyle\sum_{(s,k)\in\B} b_{s,k}$ for some index set $\B\subset \N\times\Z$
\item $b_{s,k}$ is supported on the dyadic cube $Q_{s,k}=[k2^s,(k+1)2^s)\cap\Z$
\item $\{Q_{s,k}:(s,k)\in\B\}$ is a disjoint collection
\item $\|b_{s,k}\|_1\leq\lambda|Q_{s,k}|=\lambda2^s$
\item $\displaystyle \sum_{(s,k)\in\B}|Q_{s,k}|\leq \frac{C}\lambda\| \varphi\|_1$ ($C$ independent of $\varphi$ and $\lambda$).
\end{itemize}
 Let $b_s=\displaystyle\sum_k b_{s,k}$.  We will divide $\displaystyle\sum_s b_s$ into two parts, splitting at the index $s(j):=\min\{s: 2^s\geq R_j\}$.
\\
\\We begin by noting $\{x: \sup_j | \varphi\ast\mu_j (x)|> 6\lambda\}\subset$
\begin{eqnarray*}
\{\sup_j |g \ast\mu_j|> \lambda\}\cup\{ \sup_j |b\ast(\mu_{j}-\nu_j)|> \lambda\} \cup\{\sup_j |\sum_{s=s(j)}^\infty b_s \ast\nu_j|> \lambda\}\cup\{\sup_j | \sum_{s=0}^{s(j)-1}b_s\ast\nu_j|> 3\lambda\}
\end{eqnarray*}
$$=  E_1\cup E_2\cup E_3\cup E_4.$$
\\ \noindent By the weak $(p,p)$ inequality (if $p<\infty$), $|E_1|\leq C\lambda^{-p}\|g\|_p^p\leq C\lambda^{-p}\|g\|_\infty^{p-1}\|g\|_1\leq C\lambda^{-1}\|\varphi\|_1$; if $p=\infty$, consider instead that $\{x: \sup_j |g \ast\mu_{j} (x)|> C_\infty\lambda\}=\emptyset$ since $\|\sup_j |g \ast\mu_{j}|\|_\infty\leq C_\infty\|g\|_\infty=C_\infty\lambda$.
\\
\\Next, $|b\ast(\mu_j-\nu_j) (x)|\leq|b|\ast|\mu_j-\nu_j|(x)$, so by the assumed weak $(1,1)$ inequality, $$|E_2|\leq |\{\sup_j |b|\ast|\mu_j-\nu_j|>\lambda\}|\leq \frac{C}\lambda \|b\|_1\leq \frac{C}\lambda\|\varphi\|_1.$$
To bound $|E_3|$, note that for $s\geq s(j)$, $\text{supp }(b_{s,k}\ast\nu_j )\subset Q_{s,k}+[-R_j,R_j]\subset Q^*_{s,k}$, an expansion of $Q_{s,k}$ by a factor of 3.  Thus
\begin{eqnarray*}
E_3\subset\displaystyle \bigcup_j \bigcup_{k\in\Z,s\geq s(j)}\text{supp }(b_{s,k}\ast\nu_j )\subset\bigcup_{k\in\Z,s\geq s(j)}Q^*_{s,k}
\end{eqnarray*}
and
\begin{eqnarray*}
|E_3|\leq\sum_{(s,k)\in\B} 3|Q_{s,k}|\leq\frac{C}\lambda\|\varphi\|_1.
\end{eqnarray*}
We have thus reduced the problem to obtaining a bound on $|E_4|$.  We will attempt this directly for heuristic purposes, and then modify our setup for the actual argument.  By Chebyshev's Inequality,
\begin{eqnarray}
\label{money}
|\{x: \sup_j |\sum_{s=0}^{s(j)-1} b_s\ast\nu_j (x)|>\lambda\}|\leq\lambda^{-2}\sum_x\sup_j |\sum_{s=0}^{s(j)-1} b_s\ast\nu_j (x)|^2\leq\lambda^{-2}\sum_j \left\|\sum_{s=0}^{s(j)-1} b_s\ast\nu_j\right\|_{\ell^2}^2 
\end{eqnarray}
\begin{eqnarray*}
= \lambda^{-2}\sum_j\sum_{\scriptsize \begin{array}{c} s_1,s_2:\\ 0\leq s_1,s_2 < s(j) \end{array}} \langle b_{s_1}\ast\nu_j ,b_{s_2}\ast\nu_j \rangle_{\ell^2}
\end{eqnarray*}
and we will use our estimate on the convolution product $\nu_j\ast\tilde\nu_j$:
%Lemma
\begin{lemma}
\label{wall}
Let $f,g\in\ell^1$ such that $\displaystyle\sum_{x\in Q_{s(j),k}}|g(x)|\leq \lambda2^{s(j)}$ for all $k$, and assume the $\nu_j$ satisfy (\ref{convo}).  Then
\begin{eqnarray*}
|\langle f\ast\nu_j ,g\ast\nu_j \rangle|\leq A_0 r_j^{-1}|\langle f,g\rangle|+10A_1 \lambda R_j^{-\epsilon} \|f\|_1.
\end{eqnarray*}
\end{lemma}
\begin{proof}
\begin{eqnarray*}
|\langle f\ast\nu_j ,g\ast\nu_j \rangle|&=&|\langle f\ast\nu_j\ast \tilde\nu_j, g\rangle|\\
&\leq& A_0r_j^{-1}|\langle f,g\rangle|+ A_1R_j^{-(1+\epsilon)}\|f\|_1 \|g\|_1.
\end{eqnarray*}
We let $f_k=f\chi(Q_{s(j),k})$ and $g_l=g\chi(Q_{s(j),l})$; note that $\|g_l\|_1\leq \lambda2^{s(j)}\leq2\lambda R_j$.  If $|k-l|>2$, then $\langle f_k\ast\nu_j,g_l\ast\nu_j\rangle=0$ as the supports are disjoint; thus
\begin{eqnarray*}
|\langle f\ast\nu_j ,g\ast\nu_j \rangle|&\leq&\sum_k\sum_{i=-2}^2 |\langle f_k \ast\nu_j, g_{k+i}\ast\nu_j\rangle|\\
&\leq& \sum_k\sum_{i=-2}^2 A_0r_j^{-1}|\langle f_k, g_{k+i}\rangle|+2A_1\lambda R_j^{-\epsilon}\|f_k\|_1\\
&\leq& A_0 r_j^{-1}|\langle f,g\rangle|+10A_1 \lambda R_j^{-\epsilon} \|f\|_1.
\end{eqnarray*}
\end{proof}
\noindent Therefore 
\begin{eqnarray*}
|\{x: \sup_j |\sum_{s=0}^{s(j)-1} b_s\ast\nu_j (x)|>\lambda\}|&\leq& \lambda^{-2}\sum_j\sum_{\scriptsize \begin{array}{c} s_1,s_2:\\ 0\leq s_1,s_2 < s(j) \end{array}}A_0 r_j^{-1}|\langle b_{s_1},b_{s_2}\rangle|+10A_1 \lambda R_j^{-\epsilon} \|b_{s_1}\|_1\\
&\leq& \lambda^{-2}\sum_j\sum_{0\leq s<s(j)} A_0 r_j^{-1}\|b_s\|_2^2+10A_1\lambda s(j)R_j^{-\epsilon}\|b_s\|_1\\
&\leq& A_0\lambda^{-2}\sum_j r_j^{-1}\|b\|_2^2+10A_1\lambda^{-1}\sum_j \log_2(2R_j)R_j^{-\epsilon}\|b\|_1.
\end{eqnarray*}
\noindent Since $r_j$ (and thus $R_j$) grows faster than any polynomial by the assumption $\sum_{j\leq k} r_j\leq Ar_k\;\forall k\in\N$, the second term is $\leq\frac C \lambda \|\varphi\|_1$ as desired.  The first term does not, however, give us that bound.  We will therefore decompose these functions further.
\\
\\ For each $j$, we decompose $b_{s,k}=b^{(j)}_{s,k}+B^{(j)}_{s,k}$, where $b^{(j)}_{s,k}=b_{s,k}\chi(|b_{s,k}|>\lambda r_j)$.  Define $b^{(j)}_s, B^{(j)}_s, b^{(j)}, B^{(j)}$ by summing over one or both indices, respectively.  Now we see that
\begin{eqnarray*}
E_4&\subset&\{ \sup_j |\sum_{s=0}^{s(j)-1} b^{(j)}_s\ast(\nu_j-\mu_j) |>\lambda\}\cup\{\sup_j |\sum_{s=0}^{s(j)-1} b^{(j)}_s\ast\mu_j|>\lambda\}\cup\{\sup_j |\sum_{s=0}^{s(j)-1} B^{(j)}_s\ast\nu_j |>\lambda\}\\
&=& E_5\cup E_6 \cup E_7.
\end{eqnarray*}
We control $E_5$ just as we controlled $E_2$, since $|b^{(j)}|\leq|b|$; and
\begin{eqnarray*}
|E_6|\leq\sum_j |\{x:|b^{(j)}\ast\mu_{j} (x)|>0\}|&\leq&\sum_j |\text{supp }\mu_{j}|\cdot|\{x: |b(x)|>\lambda r_j\}|\\
&=&\sum_j r_j \sum_{k\geq j} |\{x:\lambda r_k<|b(x)|\leq \lambda r_{k+1}\}|\\
&=&\sum_{k} |\{x:\lambda r_k<|b(x)|\leq \lambda r_{k+1}\}| \sum_{j\leq k}r_j\\
&\leq&\frac{A}\lambda \sum_k \lambda r_k |\{x:\lambda r_k<|b(x)|\leq \lambda r_{k+1}\}|;
\end{eqnarray*}
 now since this sum is a lower sum for $|b|$, we have $|E_6|\leq \frac A\lambda\|b\|_1\leq\frac{C}\lambda\|\varphi\|_1$.
 \\
 \\ We proceed with $E_7$ just as we tried before, since Lemma \ref{wall} applies to the $B^{(j)}_s$ as well as to the $b_s$.  We thus find
\begin{eqnarray*}
|E_7|&\leq& A_0\lambda^{-2}\sum_j r_j^{-1}\|B^{(j)}\|_2^2+10A_1\lambda^{-1}\sum_j \log_2(2R_j)R_j^{-\epsilon}\|B^{(j)}\|_1\\
&\leq& A_0\lambda^{-2}\sum_x \sum_j r_j^{-1}|B^{(j)}(x)|^2+\frac C\lambda\|\varphi\|_1.
\end{eqnarray*}
Because $\displaystyle\sum_{j\leq k} r_j\leq Ar_k\;\forall k\in\N\implies \exists N \text{ s.t. } r_{j+n}\geq 2r_j \forall j\in\N,n\geq N\implies \sum_{j=k}^\infty r_j^{-1}\leq A'r_k^{-1},$ for each $x$
\begin{eqnarray*}
\sum_j r_j^{-1}|B^{(j)}(x)|^2\leq\sum_{ j: \lambda r_j\geq |b(x)|}r_j^{-1}|b(x)|^2\leq A'\lambda|b(x)|
\end{eqnarray*}
so $|E_7|\leq \frac C\lambda\|\varphi\|_1$ and the proof of Proposition \ref{shine on} is complete.
\end{proof}

%Concluding Section
\section{Probabilistic Lemma, Conclusion of the Proof}
\noindent Having established Proposition \ref{shine on}, it remains to show that the random measures $\mu_j^{(\omega)}$ and $\nu_j^{(\omega)}$ satisfy the assumptions with probability 1.  Note first that  $r_j=|\text{supp }\mu_j^{(\omega)}|=\sum_{1\leq n\leq 2^j}\xi_n(\omega)={\mathbf\Theta}(\beta(2^j))={\mathbf\Theta}(2^{(1-\alpha)j})$ on $\Omega_1$, and $R_j=2^{j+1}$. We must prove the bound (\ref{convo}) on $\nu_j^{(\omega)}\ast \tilde\nu_j^{(\omega)}$.
%BASIC LEMMA
\begin{lemma}
\label{cancels}
Let $E\subset\Z$, and let $\{  X_n\}_{n\in E}$ be independent random variables with $|X_n|\leq1$ and $\E X_n=0$.  Assume that $\sum_{n\in E} (\Var X_n)^2\geq1$.  Let $X$ be the random $\ell^1$ function $\sum_{n\in E} X_n\delta_n$, and let $\Z^\times$ denote $\Z\setminus\{0\}$.  Then for any $\theta>0$,
\begin{eqnarray}
\P\left(\| X\ast\tilde X\|_{\ell^\infty(\Z^\times)}\geq \theta(\sum_{n=1}^N (\Var X_n)^2)^{1/2}\right)\leq 4|E|^2\max(e^{-\theta^2/16},e^{-\theta/4}).
\end{eqnarray}
\end{lemma}
\begin{proof}
For $k\neq0$,
$$X\ast\tilde X(k)=\sum_{n\in E \cap E-k}X_nX_{n+k}=\sum_{n\in E} Y_n$$
where $\E Y_n=0$ and $|Y_n|\leq1$ (of course $Y_n\equiv0$ if $n+k\notin E$).  We want to apply Chernoff's Inequality (Theorem \ref{chernoff}), but the $Y_n$ are not independent.
\\
\\However, we can easily partition $E$ into two subsets $E_1$ and $E_2$ such that $E_i\cap(E_i-k)=\emptyset$ for each $i$; then within each $E_i$, the $Y_n$ depend on distinct independent random variables, so they are independent.
\\
\\ Now $\displaystyle \sum_{n\in E_i} Y_n$ has variance $\displaystyle \sigma_i^2=\sum_{n\in E_i}\Var X_n \Var X_{n+k}\leq\sum_{n\in E}(\Var X_n)^2$
by H\"older's Inequality.  Chernoff's Inequality states that for any $\lambda>0$, $\displaystyle \P(|\sum_{n\in E_i} Y_n|\geq\lambda\sigma_i)\leq2\max(e^{-\lambda^2/4},e^{-\lambda\sigma_i/2})$.
\\
\\ Take $\lambda_i=\theta\sigma_i^{-1}(\sum_{n\in E} (\Var X_n)^2)^{1/2}$; then $\lambda_i\geq\theta$ and $\lambda_i\sigma_i=\theta(\sum_{n\in E} (\Var X_n)^2)^{1/2}\geq\theta$, so
$$\P(|X\ast\tilde X(k)|\geq 2\theta(\sum_{n\in E} (\Var X_n)^2)^{1/2})\leq\sum_{i=1}^2\P(|\sum_{n\in E_i} Y_n|\geq\lambda_i\sigma_i)\leq4\max(e^{-\theta^2/4},e^{-\theta/2}).$$
Since this holds for each $k\neq 0$ and $|\text{supp }X\ast\tilde X|\leq |E|^2$, the conclusion follows (replacing $2\theta$ with $\theta$).
\end{proof}
%corollary
\begin{corollary}
\label{bound}
Let $\nu_j^{(\omega)}$ be the random measure defined as before, $0<\alpha<1/2$ and $\kappa>0$.  Then there is a set $\Omega_3\subset\Omega_2$ with $\P(\Omega_3=1)$ such that for each $\omega\in\Omega_3$,
\begin{eqnarray}
|\nu_j^{(\omega)}\ast \tilde\nu_j^{(\omega)}(x)|\leq C_\omega\beta(2^j)^{-1}\delta_0(x)+C_\omega\beta(2^j)^{-2}2^{\kappa j}\sqrt{\sum_{n=1}^{2^j}\tau_n^2}.
\end{eqnarray}
\end{corollary}
\begin{proof}
For the bound at $0$, we use the fact that 
\begin{eqnarray*} \nu_j^{(\omega)}\ast\tilde\nu_j^{(\omega)}(0)=\beta(2^j)^{-2}\sum_{n=1}^{2^j}(\xi_n(\omega)-\tau_n)^2&=&\beta(2^j)^{-2}\sum_{n=1}^{2^j}\left(\tau_n^2(1-\xi_n(\omega))+(1-\tau_n)^2\xi_n(\omega)\right)\\
&\leq&\beta(2^j)^{-2}\sum_{n=1}^{2^j}(\tau_n+\xi_n(\omega))=2\beta(2^j)^{-1}+\beta(2^j)^{-2}\sum_{n=1}^{2^j}(\xi_n(\omega)-\tau_n)
\end{eqnarray*}
so that
$$\P(\nu_j^{(\omega)}\ast\tilde\nu_j^{(\omega)}(0)>3\beta(2^j)^{-1})\leq\P\left(\sum_{n=1}^{2^j} (\xi_n(\omega)-\tau_n) >\beta(2^j)\right)\leq 2\exp(-\frac12\beta(2^j))$$ for $j$ sufficiently large, by Chernoff's inequality.  The Borel-Cantelli Lemma implies that $\nu_j^{(\omega)}\ast\tilde\nu_j^{(\omega)}(0)\leq 3\beta(2^j)^{-1}$ for $j$ sufficiently large (depending on $\omega$), so there exists $C_\omega$ with $0\leq\nu_j^{(\omega)}\ast\tilde\nu_j^{(\omega)}(0)\leq C_\omega \beta(2^j)^{-1}$ for all $j$.
\\
\\ For the other term, we note that $\Var \xi_n\leq \tau_n$, so we set $\theta=2^{\kappa j}$ and apply Lemma \ref{cancels}:
\begin{eqnarray*}
\P\left(\beta(2^j)^2\|\nu_j^{(\omega)}\ast\tilde\nu_j^{(\omega)}\|_{\ell^\infty(\Z^\times)}\geq 2^{\kappa j}(\sum_{n=1}^{2^j}\tau_n^2)^{1/2}\right)\leq 4\cdot2^{2j}\exp(-2^{\kappa j}/4)
\end{eqnarray*}
which sum over $j$.  The Borel-Cantelli Lemma again proves the bound holds with probability 1.
\end{proof}
\noindent Note that $\displaystyle\sum_{n=1}^{2^j}\tau_n^2={\mathbf\Theta}(2^{(1-2\alpha)j});$ thus for $\alpha<1/2$ and $\kappa+\epsilon=1/2-\alpha$, 
$$\beta(2^j)^{-2}(\sum_{n=1}^{2^j}\tau_n^2)^{1/2}2^{\kappa j}=O(2^{(-\frac{3}2+\alpha+\kappa)j})=O(R_j^{-(1+\epsilon)}).$$
Therefore the measures $\nu_j^{(\omega)}$ satisfy the bound (\ref{convo}) , for all $\omega\in\Omega_3$.  Since $\mu_j^{(\omega)}-\nu_j^{(\omega)}=\E\mu_j$ is a weighted average of the regular ergodic averages, $\sup_j |\varphi\ast\E\mu_j|\leq C\sup_N|\varphi\ast N^{-1}\chi[1,N]|$ so that Birkhoff's Ergodic Theorem implies the needed weak $\ell^1$ bound; and the $\ell^\infty$ maximal inequality for $\mu_j^{(\omega)}$ is trivial.  Thus Proposition \ref{shine on} implies the weak maximal inequality (\ref{eclipse}), and we have proved Theorem \ref{L1}.
\begin{remark}
This argument does not require $\tau_n$ to obey a power law.  If $\tau_n$ is decreasing and if $\displaystyle\beta(2^j)^{-2}\sqrt{\sum_{n=1}^{2^j}\tau_n^2}\leq C2^{-(1+\epsilon)j}$ for some $\epsilon>0, C<\infty$ and all $j$, the sequence $\{n:\xi_n(\omega)=1\}$ will be universally $L^1$-good with probability 1.
\end{remark}
It remains, finally, to note that $\{n:\xi_n=1\}$ indeed has Banach density 0 (with probability 1) if the $\tau_n$ decrease more rapidly than some power law.  Conveniently enough, a converse result also holds:
\begin{proposition}
\label{banach}
Let $\{\tau_n\}$ be a decreasing sequence of probabilities, and let $\xi_n$ be independent Bernoulli random variables with $\P(\xi_k=1)=k^{-\alpha}$.  Then if $\tau_n= O(n^{-\alpha})$ for some $\alpha>0$, the sequence of integers $\{n:\xi_n=1\}$ has Banach density 0 with probability 1 in $\Omega$; otherwise, it has Banach density 1 with probability 1 in $\Omega$.
\end{proposition}
\begin{proof}
It is elementary to show that 
\begin{eqnarray}
\label{easy}
2^{-r}\tau_{r(n+1)}^m\leq\P\left( \sum_{j=rn}^{r(n+1)-1}\xi_j \geq m\right)\leq2^r\tau_{rn}^m.
\end{eqnarray}
(We majorize or minorize the $\xi_j$ by i.i.d. Bernoulli variables and use the Binomial Theorem.) Then if $\tau_n= O(n^{-\alpha})$, let $K>0$ and fix $m,r\in\N$ such that $m\alpha>1$ and $r>mK$; the probabilities above are then summable, so the first Borel-Cantelli Lemma implies that on a set $\Omega_K$ of probability 1 in $\Omega$, there exists an $M_\omega$ such that for all $n\geq M_\omega$, $\sum_{j=rn}^{r(n+1)-1}\xi_j < m<\frac{r}{K}$; then it is clear that $\{n:\xi_n=1\}$ has Banach density less than $3K^{-1}$.  Let $\Omega'=\bigcap_K \Omega_K$; then $\P(\Omega')=1$ and $\{n:\xi_n=1\}$ has Banach density 0 on $\Omega'$.
\\
\\ For the other implication, note that if $\tau_n\neq O(n^{-1/R})$, there exists a sequence $n_k$ with $n_{k+1}\geq2n_k$ such that $\tau_{n_k}\geq n_k^{-1/R}$; then $$\sum_{n=1}^\infty \tau_{Rn}^R\geq R^{-1}\sum_{n=2}^\infty \tau_n^R\geq R^{-1}\sum_{k=2}^\infty(n_k-n_{k-1})\tau_{n_k}^R\geq R^{-1}\sum_{k=2}^\infty \frac12=\infty.$$
Thus the probabilities in (\ref{easy}) are not summable in $n$, for $m=r=R$.  Since the variables $\xi_n$ are independent, the second Borel-Cantelli Lemma implies that there is a set $\tilde\Omega_R$ of probability 1 on which $\{n:\xi_n(\omega)=1\}$ contains infinitely many blocks of $R$ consecutive integers.  Therefore if $\tau(n)\neq O(n^{-\alpha})$ for every $\alpha>0$, let $\tilde\Omega'=\bigcap_R \tilde\Omega_R$; on this set of probability 1, $\{n:\xi_n=1\}$ has Banach density 1.
\end{proof}
\noindent The author thanks his dissertation advisor, M. Christ, for consultation and assistance throughout the composition of this paper, and M. Wierdl and J. Rosenblatt for many comments and suggestions.

\end{document}